\documentclass[11pt]{amsart}

\usepackage[margin=1.1in]{geometry}
\usepackage{amsmath,amssymb,amsthm,mathtools}
\usepackage{booktabs}
\usepackage{enumitem}
\usepackage{tikz}
\usetikzlibrary{arrows.meta,positioning,calc,decorations.pathreplacing,shapes.misc}
\usepackage{pgfplots}
\pgfplotsset{compat=1.18}
\usepackage[colorlinks=true,linkcolor=blue!50!black,citecolor=green!40!black,urlcolor=blue!50!black]{hyperref}
\usepackage[capitalise,nameinlink]{cleveref}

\newtheorem{theorem}{Theorem}[section]
\newtheorem{proposition}[theorem]{Proposition}
\newtheorem{lemma}[theorem]{Lemma}
\newtheorem{corollary}[theorem]{Corollary}
\newtheorem{conjecture}[theorem]{Conjecture}
\theoremstyle{definition}
\newtheorem{definition}[theorem]{Definition}
\newtheorem{example}[theorem]{Example}
\theoremstyle{remark}
\newtheorem{remark}[theorem]{Remark}
\numberwithin{equation}{section}
\crefname{conjecture}{Conjecture}{Conjectures}
\crefname{corollary}{Corollary}{Corollaries}

\newcommand{\PF}{\mathrm{PF}_\infty}
\newcommand{\Lop}{\mathcal{L}_n}
\newcommand{\Fop}{\mathcal{F}}
\newcommand{\hp}{\odot}
\newcommand{\half}{\tfrac12}
\DeclareMathOperator{\sgn}{sgn}

\begin{document}

\title[Zeros of Toeplitz determinants of rising factorial series]{A proof of Karp's conjecture on the zeros of\\ Toeplitz determinants of rising factorial series}

\author{Henry Zweiman}
\address{Brown University, Providence, RI 02912, USA}
\email{henry\_zweiman@brown.edu}


\begin{abstract}
 In 2012, D.~Karp studied a family of polynomials constructed from P\'olya frequency sequences and conjectured that all their zeros are real and negative. We prove this conjecture for positive sequences and establish further results concerning the location of these zeros.
\end{abstract}

\maketitle

\section{Introduction}\label{sec:intro}

For a sequence of real numbers $(a_k)_{k\ge0}$, the \emph{rising factorial series}
\[
f(x;z)=\sum_{k\ge0} a_k\,(x)_k\,\frac{z^k}{k!},\qquad (x)_0=1,\quad (x)_k=x(x+1)\cdots(x+k-1),
\]
is a formal power series in $z$ whose coefficients are polynomials in the parameter $x$. The log-concavity of such series with respect to $x$, and the corresponding Tur\'an-type inequalities, have been studied extensively; see \cite{BGR,KK13,Ka12,KS,KZ23}.

Karp \cite{Ka12} considered the Toeplitz determinant
\[
\begin{vmatrix} f(x;z) & f(x+1;z)\\ f(x-1;z) & f(x;z)\end{vmatrix}
= f(x;z)^2-f(x-1;z)\,f(x+1;z)=\sum_{n\ge0}P_n(x)\,\frac{z^n}{n!}.
\]
Multiplying out the series gives\footnote{In \cite{Ka12}, the coefficients are denoted $f_k$, and $P_n(x)=Q_n^{1,1}(x-1)$ in the notation used there.}
\begin{equation}\label{eq:Pn}
P_n(x)=\sum_{k=0}^n\binom nk a_ka_{n-k}\Bigl[(x)_k(x)_{n-k}-(x+1)_k(x-1)_{n-k}\Bigr].
\end{equation}
Karp proved \cite[Theorem~1]{Ka12} that if $a_k^2>a_{k-1}a_{k+1}$ for $1\le k\le n-1$, then $P_n$ has degree $n-2$ and positive coefficients.

\begin{definition}
A sequence $(a_k)_{k\ge0}$ is a \emph{P\'olya frequency sequence of infinite order} ($\PF$) if all minors of the infinite Toeplitz matrix $(a_{j-i})_{i,j\ge0}$, where $a_k=0$ for $k<0$, are nonnegative \cite{Karlin}. A finite sequence $(a_0,\dots,a_n)$ is $\PF$ if its extension by zeros is $\PF$. Equivalently, a finite nonnegative sequence $(a_0,\dots,a_n)$ is $\PF$ if and only if its generating polynomial $\sum_{k=0}^n a_kz^k$ has only real, nonpositive zeros, not identically zero \cite{ASW}.
\end{definition}

Karp formulated the following conjecture.

\begin{conjecture}[{Karp \cite[Conjecture~3]{Ka12}}]\label[conjecture]{conj:karp}
If $\{a_k\}_{k=0}^n$ is a $\PF$ sequence and $n\ge3$, then all zeros of $P_n(x)$ are real and negative.
\end{conjecture}

The real-rootedness characterization in the definition is how the conjecture is phrased operationally in \cite[\S\S4--5]{Ka12}, and how the numerical evidence there was generated. Karp also reported counterexamples showing that the conclusion fails for merely log-concave sequences, and for the analogous polynomials $Q_n^{\alpha,\beta}$ with general shifts $\alpha,\beta>0$ \cite[\S1]{Ka12}. Karp restated the conjecture for positive sequences in 2023 \cite[Karp, \S1, Conjecture~1]{Kry}.

Our main result settles \cref{conj:karp}, with explicit information on the location of the roots.

\begin{theorem}\label{thm:main}
Let $n\ge2$ and let $a_0,\dots,a_n>0$ be such that $p(z)=\sum_{k=0}^na_kz^k$ has only real roots. Then:
\begin{enumerate}[label=\textup{(\alph*)}]
\item $P_n$ has degree exactly $n-2$ and a positive leading coefficient.
\item $P_n$ factors as
\[
P_n(x)=a_0a_n\,2^{n-2}\,(x+1)_{\lceil n/2\rceil-1}\,R(x),
\]
where the polynomial $R$, given explicitly in \eqref{eq:R0}, has degree $\lfloor n/2\rfloor-1$, a positive leading coefficient, and only simple real roots, all of which lie in $\bigl(-\lfloor n/2\rfloor+\half,\,0\bigr)$ and are pairwise more than one unit apart.
\item Consequently, all roots of $P_n$ are real and lie in the interval $\bigl[-\tfrac{n-1}{2},\,0\bigr)$. The left endpoint is a root if and only if $n$ is odd.
\end{enumerate}
In particular, \cref{conj:karp} holds for all positive $\PF$ sequences. When $a_0a_n=0$ the zeros are still real and nonpositive, but $0$ can occur (\cref{prop:degenerate}), so the conjecture as stated requires $a_1a_{n-1}\neq a_0a_n$.
\end{theorem}

\begin{remark}\label{rem:sign}
The positivity assumption is only a normalization. Since $P_n$ is a quadratic form in $(a_0,\dots,a_n)$, replacing every $a_k$ by $\lambda a_k$ with $\lambda\neq0$ multiplies $P_n$ by $\lambda^2>0$. If $p$ is any real polynomial of degree $n$ whose roots are all real and negative, then its coefficients are nonzero and all have the sign of $a_n$, and dividing $p$ by $a_n$ reduces to the case of positive coefficients. Hence \cref{thm:main} holds verbatim for every such $p$.
\end{remark}

Because a polynomial with positive leading coefficient and only negative roots has positive coefficients, and because of the Aissen--Schoenberg--Whitney theorem \cite{ASW}, we also obtain:

\begin{corollary}\label[corollary]{cor:PF}
Under the hypotheses of \cref{thm:main}, the coefficient sequence of $P_n$ is a positive $\PF$ sequence. Equivalently, the nonlinear map sending $\{a_k\}_{k=0}^n$ to the coefficient sequence of $P_n$ preserves positive $\PF$ sequences. Moreover, $Q_n^{1,1}(y)=P_n(y+1)$ has all its roots in $\bigl[-\tfrac{n+1}{2},-1\bigr)$; in particular it is Hurwitz stable.
\end{corollary}

The first statement of \cref{cor:PF} is the reformulation of \cref{conj:karp} given in \cite[\S4]{Ka12}; it strengthens \cite[Theorem~1]{Ka12} for $\PF$ sequences. The last statement is the case $\alpha=\beta=1$ of Karp's stability conjecture \cite[Conjecture~2]{Ka12} for $\PF$ sequences (positive $\PF$ sequences are strictly log-concave by Newton's inequalities \cite{HLP}). Nonlinear operators on sequences that preserve $\PF$, or equivalently real-rootedness, have been studied by Br\"and\'en \cite{Br}; the operator in \cref{cor:PF} is of a different character, as discussed in \cite[\S4]{Ka12}.

\subsection*{Organization}
\Cref{sec:prelim} recalls two identities for rising factorials and the strict Hadamard product theorem. \Cref{sec:proof} contains the proof of \cref{thm:main} and \cref{cor:PF}, together with a picture of the pairing step and explicit low-degree formulas. \Cref{sec:remarks} treats the case of $\PF$ sequences with vanishing end terms and contains concluding remarks.

\section{Preliminaries}\label{sec:prelim}

\subsection{Rising factorials}
For an indeterminate (or complex number) $x$ and an integer $k\ge0$, define the \emph{rising factorial} by
\[
(x)_0=1,
\qquad
(x)_k=x(x+1)\cdots(x+k-1) \quad (k\ge1).
\]
\begin{lemma}[Splitting identity for rising factorials]\label{lem:split}
For all integers $k,l\ge0$,
\[
(x)_{k+l}=(x)_k\,(x+k)_l .
\]
\end{lemma}

We will also use the following form of the Vandermonde identity.
\begin{lemma}[Vandermonde identity for rising factorials]\label{lem:vandermonde}
For every integer $N\ge0$ and indeterminates $u,v$,
\[
\sum_{\ell=0}^N\binom N\ell(u)_\ell\,(v)_{N-\ell}=(u+v)_N .
\]
\end{lemma}

\begin{proof}
Both sides are polynomials in $u$ and $v$, so it suffices to check the identity for nonnegative integers $u,v$. In this case, we have as formal power series, \[\sum_{k\ge0}(u)_k z^k/k!=(1-z)^{-u}.\] Comparing the coefficients of $z^N/N!$ in $(1-z)^{-u}(1-z)^{-v}=(1-z)^{-(u+v)}$ gives the claim. See also \cite[Ch.~2]{AAR}.
\end{proof}

\subsection{The strict Hadamard product theorem}
\begin{definition}
The \emph{Hadamard product} of two polynomials
$p_1(z)=\sum_k\xi_kz^k$ and $p_2(z)=\sum_k\eta_kz^k$ is their
coefficient-wise product
\[
(p_1\hp p_2)(z)=\sum_k\xi_k\eta_kz^k.
\]
\end{definition}

\begin{lemma}[Strict Hadamard product theorem; {Garloff--Wagner \cite[Theorem~4(a)]{GW}}]\label{lem:hadamard}
Let $p_1,p_2$ be real polynomials with positive leading coefficients and only nonpositive roots. If $p_1\hp p_2\neq0$, then all roots of $p_1\hp p_2$ are real and nonpositive, and they are simple, except possibly for a root at the origin.
\end{lemma}

The non-strict version is classical \cite{Malo,Schur}; see also \cite{CC,Obr,Wagner}.

\section{Proof of the main theorem}\label{sec:proof}

Throughout this section, let $n\ge2$ and write
\[
n=2m+\varepsilon,\quad \varepsilon\in\{0,1\},\quad d=m-1,
\qquad\text{so that}\qquad
\left\lceil\frac n2\right\rceil-1=m-1+\varepsilon.
\]
Let $a_0,\dots,a_n$ be positive numbers such that
$p(z)=\sum_{k=0}^n a_kz^k$ has only real roots, which are therefore
negative, and define
\begin{equation}\label{eq:C}
C(z)=\sum_{k=0}^n c_kz^k,\qquad
c_k=\binom nk a_ka_{n-k};\qquad
\text{in particular, }c_0=a_0a_n.
\end{equation}
Finally, let $\Lop$ be the linear map from polynomials of degree at most
$n$ in $z$ to polynomials in $x$ determined by
\[
\Lop(z^k)=(x)_k(x)_{n-k}-(x+1)_k(x-1)_{n-k},\qquad 0\le k\le n.
\]
Then \eqref{eq:Pn} and \eqref{eq:C} give
\[
P_n=\Lop(C).
\]

\begin{lemma}\label[lemma]{lem:Croots}
$C$ has degree $n$, positive coefficients, and only simple negative roots.
\end{lemma}

\begin{proof}
All $c_k$ are positive and $c_n=a_na_0>0$. The polynomial $z^np(1/z)=\sum_ka_{n-k}z^k$ just reverses the coefficients of $p$; its roots are the reciprocals $1/r_i$ of the roots $r_i$ of $p$, so they are real and negative. By \cref{lem:hadamard}, the polynomial
\[
p(z)\hp z^np(1/z)=\sum_{k=0}^na_ka_{n-k}z^k
\]
has only nonpositive roots. Similarly $(1+z)^n$ has only negative roots, so by \cref{lem:hadamard} again, all roots of
\[
p(z)\hp z^np(1/z)\hp(1+z)^n=\sum_{k=0}^n\binom nka_ka_{n-k}z^k=C(z)
\]
are nonpositive and simple, except possibly at the origin. Since $C(0)=c_0>0$, the origin is not a root, so all roots of $C$ are negative and simple.
\end{proof}

We now rewrite $C$ in a very specific form, chosen so that $\Lop$ can be evaluated on it in closed form. First, $C$ is palindromic: by definition, $c_k=c_{n-k}$, that is,
\[C(z)=z^nC(1/z).\]
By \cref{lem:Croots}, the $n$ roots of $C$ are distinct negative numbers, and the leading coefficient of $C$ is $c_n=c_0$. If $C(-\tau)=0$ with $\tau>0$, then because $C$ is palindromic
\[
C(-1/\tau)=(-1/\tau)^nC(-\tau)=0.
\]
Hence the roots of $C$ different from $-1$ split into disjoint pairs $\{-s,-1/s\}$ with $s>0$, $s\neq1$.  The number of these roots is even, and $-1$ is a root at most once; since $C$ has exactly $n=2m+\varepsilon$ roots, $-1$ is a root precisely when $\varepsilon=1$, and there are exactly $m$ pairs. Writing them as $\{-s_j,-1/s_j\}$, $1\le j\le m$, with $s_j>0$ and $s_j\neq1$, we obtain
\begin{equation}\label{eq:pairs}
C(z)=c_0\,(z+1)^\varepsilon\prod_{j=1}^m(z+s_j)\Bigl(z+\frac1{s_j}\Bigr).
\end{equation}

\begin{lemma}\label{lem:pairing}
The numbers $b_j=s_j+\frac1{s_j}-2$, $1\le j\le m$, are positive and pairwise distinct, and
\begin{equation}\label{eq:pairs2}
C(z)=c_0\,(z+1)^\varepsilon\prod_{j=1}^m\Bigl((z+1)^2+b_jz\Bigr).
\end{equation}
\end{lemma}

\begin{proof}
Expanding $(z+s_j)(z+\frac{1}{s_j})=(z+1)^2+b_jz$ in \eqref{eq:pairs} gives \eqref{eq:pairs2}, and positivity of $b_j$ follows from the arithmetic-geometric mean inequality on $s+1/s$. That is, $s_j+1/s_j\ge2$, with equality only for $s_j=1$; so $b_j>0$. Finally, if $b_j=b_k$ for some $j\neq k$, then $s_k+1/s_k=s_j+1/s_j$, which forces $s_k\in\{s_j,1/s_j\}$. Therefore the pairs $\{-s_j,-1/s_j\}$ and $\{-s_k,-1/s_k\}$ coincide and $C$ would have a double root, contradicting \cref{lem:Croots}. Thus the numbers $b_j$ are pairwise distinct.
\end{proof}

Now we introduce the polynomial
\[g(t)=\prod_{j=1}^m(1+b_jt)=\sum_{j=0}^mg_jt^j,\]
which for the moment serves to collect coefficients, but will be used once more in \cref{sec:h}. Each $g_j$ is the $j$-th elementary symmetric polynomial in $b_1,\dots,b_m$. In terms of $g$, $C$ becomes
\begin{equation}\label{eq:Cg}
C(z)=c_0\sum_{j=0}^mg_j\,z^j(z+1)^{n-2j}.
\end{equation}
Indeed, expand the product in \eqref{eq:pairs2} by choosing, in each factor, either $(z+1)^2$ or $b_jz$. The terms in which $b_jz$ is chosen for exactly $j$ of the $m$ factors add up to $g_j\,z^j(z+1)^{2(m-j)}$, and multiplying by the front factor $(z+1)^\varepsilon$ turns the exponent into $2(m-j)+\varepsilon=n-2j$.

\begin{remark}\label{rem:folding}
Since $z^j(z+1)^{n-2j}=(z+1)^n\bigl(z/(z+1)^2\bigr)^j$, identity \eqref{eq:Cg} can be written as
\[
C(z)=c_0\,(z+1)^n\,g\!\left(\frac{z}{(z+1)^2}\right).
\]
The substitution $t=z/(z+1)^2$ sends both roots $-s_j$ and $-1/s_j$ of $C$ to the same point $-1/(s_j+s_j^{-1}-2)=-1/b_j$, which is the root of the factor $1+b_jt$ of $g$; see \cref{fig:folding}. On $(-\infty,0)\setminus\{-1\}$ this substitution is two-to-one onto $(-\infty,0)$, and the root $-1$ of $C$ (present exactly when $n$ is odd) corresponds to the front factor $(z+1)^\varepsilon$ in \eqref{eq:pairs}.
\end{remark}

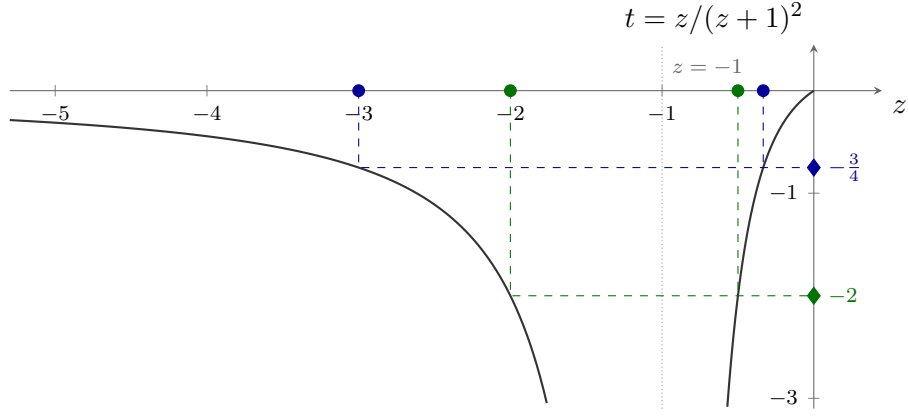
\begin{figure}[t]
\centering
\begin{tikzpicture}
\begin{axis}[
  width=0.82\linewidth, height=6.4cm,
  xmin=-5.3, xmax=0.45, ymin=-3.1, ymax=0.45,
  axis lines=middle, axis line style={black!60},
  xlabel={$z$}, ylabel={$t=z/(z+1)^2$},
  xlabel style={at={(ticklabel* cs:1)},anchor=north west},
  ylabel style={at={(ticklabel* cs:1)},anchor=south east},
  xtick={-5,-4,-3,-2,-1}, ytick={-3,-1},
  ticklabel style={font=\scriptsize},
  samples=160, clip=true]
\addplot[black!80, thick, domain=-5.3:-1.76] {x/(1+x)^2};
\addplot[black!80, thick, domain=-0.57:0] {x/(1+x)^2};
\draw[black!40, densely dotted] (axis cs:-1,-3.1) -- (axis cs:-1,0.45);
\draw[blue!60!black, dashed] (axis cs:-3,0) -- (axis cs:-3,-0.75) -- (axis cs:0,-0.75);
\draw[blue!60!black, dashed] (axis cs:-0.3333,0) -- (axis cs:-0.3333,-0.75);
\draw[green!45!black, dashed] (axis cs:-2,0) -- (axis cs:-2,-2) -- (axis cs:0,-2);
\draw[green!45!black, dashed] (axis cs:-0.5,0) -- (axis cs:-0.5,-2);
\addplot[only marks, mark=*, mark size=2.2pt, blue!60!black] coordinates {(-3,0) (-0.3333,0)};
\addplot[only marks, mark=*, mark size=2.2pt, green!45!black] coordinates {(-2,0) (-0.5,0)};
\addplot[only marks, mark=diamond*, mark size=3.2pt, blue!60!black] coordinates {(0,-0.75)};
\addplot[only marks, mark=diamond*, mark size=3.2pt, green!45!black] coordinates {(0,-2)};
\node[font=\scriptsize, blue!60!black, anchor=west] at (axis cs:0.03,-0.75) {$-\tfrac34$};
\node[font=\scriptsize, green!45!black, anchor=west] at (axis cs:0.03,-2) {$-2$};
\node[font=\scriptsize, black!60, anchor=south west] at (axis cs:-1,0.05) {$z=-1$};
\end{axis}
\end{tikzpicture}
\caption{The substitution $t=z/(z+1)^2$ on the negative axis. The pairs of roots $\{-s_1,-1/s_1\}=\{-3,-\tfrac13\}$ and $\{-s_2,-1/s_2\}=\{-2,-\tfrac12\}$ of $C$ (dots on the $z$-axis) are sent to the roots $-1/b_1=-\tfrac34$ and $-1/b_2=-2$ of $g$ (diamonds on the $t$-axis), where $b_j=s_j+s_j^{-1}-2$.}
\label{fig:folding}
\end{figure}

We now evaluate $\Lop$ on $C$. To do so, we introduce a second linear map $\Fop$. For a polynomial $\varphi(z)=\sum_{k=0}^n\varphi_kz^k$ and indeterminates $u,v$ put
\begin{equation}\label{eq:F}
\Fop(\varphi;u,v)=\sum_{k=0}^n\varphi_k\,(u)_k\,(v)_{n-k},\qquad\text{so that}\qquad \Lop(\varphi)=\Fop(\varphi;x,x)-\Fop(\varphi;x+1,x-1).
\end{equation}
By linearity and \eqref{eq:Cg}, it suffices to compute $\Fop$ on each piece $z^j(z+1)^{n-2j}$, for arbitrary $u,v$. Afterwards we substitute $(u,v)=(x,x)$ and $(u,v)=(x+1,x-1)$. This will give us $\Lop(C)$.

\begin{lemma}\label{lem:F}
For $0\le j\le m$,
\[
\Fop\bigl(z^j(z+1)^{n-2j};u,v\bigr)=(u)_j\,(v)_j\,(u+v+2j)_{n-2j}.
\]
\end{lemma}

\begin{proof}
Since $z^j(z+1)^{n-2j}=\sum_{\ell=0}^{n-2j}\binom{n-2j}{\ell}z^{j+\ell}$, we have
\begin{align*}
\Fop\bigl(z^j(z+1)^{n-2j};u,v\bigr)&=\sum_{\ell=0}^{n-2j}\binom{n-2j}{\ell}(u)_{j+\ell}(v)_{n-j-\ell}\\
&=(u)_j(v)_j\sum_{\ell=0}^{n-2j}\binom{n-2j}{\ell}(u+j)_\ell(v+j)_{n-2j-\ell}\\
&=(u)_j(v)_j(u+v+2j)_{n-2j}.
\end{align*}
The second equality pulls out a common factor using the splitting identity from \cref{lem:split}, using the forms $(u)_{j+\ell}=(u)_j(u+j)_\ell$ and $(v)_{n-j-\ell}=(v)_j(v+j)_{n-2j-\ell}$; the third is \cref{lem:vandermonde} with $N=n-2j$.
\end{proof}

Since $u+v=2x$ for both substitutions, \cref{lem:F} gives
\begin{equation}\label{eq:Fsub}
\begin{aligned}
\Fop\bigl(z^j(z+1)^{n-2j};x,x\bigr)&=(x)_j^2\,(2x+2j)_{n-2j},\\
\Fop\bigl(z^j(z+1)^{n-2j};x+1,x-1\bigr)&=(x+1)_j(x-1)_j\,(2x+2j)_{n-2j}.
\end{aligned}
\end{equation}

\subsection{A new formula for \texorpdfstring{$P_n$}{P\_n}}\label{sec:formula}

Using $\Fop$ on each $z^j(z+1)^{n-2j}$ and linearity, we can construct $P_n=\Lop(C)$ in a form that allows us to break it into more approachable pieces. By linearity, \eqref{eq:Cg}, \eqref{eq:F} and \eqref{eq:Fsub},
\begin{equation}\label{eq:P1}
\begin{aligned}
P_n(x)=\Lop(C)&=c_0\sum_{j=0}^mg_j\,\Lop\bigl(z^j(z+1)^{n-2j}\bigr)\\
&=c_0\sum_{j=0}^mg_j\,(2x+2j)_{n-2j}\Bigl[(x)_j^2-(x+1)_j(x-1)_j\Bigr].
\end{aligned}
\end{equation}
The bracket vanishes for $j=0$. For $j\ge1$ its two terms share the large factor $(x)_{j-1}(x+1)_{j-1}$: using $(x)_j=x\,(x+1)_{j-1}=(x)_{j-1}(x+j-1)$, alongside $(x+1)_j=(x+1)_{j-1}(x+j)$ and $(x-1)_j=(x-1)(x)_{j-1}$, we get
\[
(x)_j^2-(x+1)_j(x-1)_j=(x)_{j-1}(x+1)_{j-1}\Bigl[x(x+j-1)-(x+j)(x-1)\Bigr]=j\,(x)_{j-1}(x+1)_{j-1}.
\]
Hence\footnote{The factor $jg_j$ suggests a derivative; this is made precise in \cref{sec:h}.}
\begin{equation}\label{eq:P2}
P_n(x)=c_0\sum_{j=1}^mj\,g_j\,(2x+2j)_{n-2j}\,(x)_{j-1}(x+1)_{j-1}.
\end{equation}

Next, we pull the factors of $2$ out of $(2x+2j)_{n-2j}$. The result rises in half-steps, so we split it into two threads, one starting at $x+j$ and the other at $x+j+\half$. Precisely, since $n-2j=2(m-j)+\varepsilon$,
\[
(2x+2j)_{n-2j}=\prod_{i=0}^{n-2j-1}(2x+2j+i)=2^{n-2j}\prod_{i=0}^{n-2j-1}\Bigl(x+j+\frac i2\Bigr),
\]
and the even indices $i=2l$ ($0\le l\le m-j+\varepsilon-1$) and odd indices $i=2l+1$ ($0\le l\le m-j-1$) give
\[(2x+2j)_{n-2j}=2^{n-2j}\,(x+j)_{m-j+\varepsilon}\,\bigl(x+j+\half\bigr)_{m-j}.\]
The thread $(x+j)_{m-j+\varepsilon}$ lines up with the factor $(x+1)_{j-1}$ in \eqref{eq:P2}: by \cref{lem:split}, $(x+1)_{j-1}(x+j)_{m-j+\varepsilon}=(x+1)_{m-1+\varepsilon}$, which no longer depends on $j$. Therefore
\[
P_n(x)=c_0\,(x+1)_{m-1+\varepsilon}\sum_{j=1}^m2^{n-2j}\,j\,g_j\,(x)_{j-1}\bigl(x+j+\half\bigr)_{m-j}.
\]
Finally we put $r=j-1$ and $d=m-1$, and write $2^{n-2j}=2^{n-2}\,4^{-r}$. This gives
\begin{equation}\label{eq:R0}
P_n(x)=c_0\,2^{n-2}\,(x+1)_{m-1+\varepsilon}\,R(x),\qquad
R(x)=\sum_{r=0}^d\frac{(r+1)\,g_{r+1}}{4^r}\,(x)_r\,\bigl(x+r+\tfrac32\bigr)_{d-r}.
\end{equation}
We have separated $P_n$ into two parts. The first factor $(x+1)_{m-1+\varepsilon}$ has the negative roots $-1,-2,\dots,-(m-1+\varepsilon)$. It remains to show that $R$ has only negative real roots.

\subsection{The polynomial \texorpdfstring{$h$}{h}}\label{sec:h}
We return to $g(t)=\prod_{j=1}^m(1+b_jt)$. Its roots are the numbers $-1/b_j$, which by \cref{lem:pairing} are $m$ distinct negative numbers $\theta_1<\dots<\theta_m<0$. By Rolle's theorem, $g'$ has a root in each interval $(\theta_i,\theta_{i+1})$, $1\le i\le m-1$. These are $m-1=d$ distinct negative roots, and $g'$ has degree exactly $m-1$ (its leading coefficient is $m\,b_1\cdots b_m\neq0$), so they are all the roots of $g'$. Rescaling the variable by a positive factor does not change the signs of the roots, so
\begin{equation}\label{eq:h}
h(t)\coloneqq g'\Bigl(\frac t4\Bigr)=\sum_{r=0}^dh_rt^r,\qquad h_r=\frac{(r+1)\,g_{r+1}}{4^r},
\end{equation}
has $d$ distinct negative roots as well. Comparing with \eqref{eq:R0}, we see that
\begin{equation}\label{eq:Rh}
R(x)=\sum_{r=0}^dh_r\,(x)_r\,\bigl(x+r+\tfrac32\bigr)_{d-r}.
\end{equation}
Since $h(0)=g_1=b_1+\dots+b_m>0$ and the roots of $h$ are distinct negative numbers $-1/u_1,\dots,-1/u_d$, we can write
\begin{equation}\label{eq:hfactor}
h(t)=h(0)\prod_{j=1}^d(1+u_jt),\qquad u_1,\dots,u_d>0\ \text{distinct}.
\end{equation}
If $d=0$, then $R=h(0)>0$ is a positive constant. From now on we assume $d\ge1$.

\subsection{The recurrence}\label{sec:recurrence}
For $0\le k\le d$ let $f_k(t)$ be the product of the first $k$ factors of the product \eqref{eq:hfactor}, and define $Q_k$ accordingly as follows:
\begin{equation}\label{eq:Qk}
f_k(t)=\prod_{j=1}^k(1+u_jt)=\sum_{r=0}^kf_{k,r}\,t^r,\qquad
Q_k(x)=\sum_{r=0}^kf_{k,r}\,(x)_r\,\bigl(x+r+\tfrac32\bigr)_{k-r}.
\end{equation}
Then $f_0=Q_0=1$, $h=h(0)f_d$, and by \eqref{eq:Rh}
\begin{equation}\label{eq:RQ}
R(x)=h(0)\,Q_d(x).
\end{equation}

\begin{lemma}[Recurrence]\label{lem:recurrence}
For $1\le k\le d$,
\begin{equation}\label{eq:Qrec}
Q_k(x)=\bigl(x+k+\half\bigr)\,Q_{k-1}(x)+x\,u_k\,Q_{k-1}(x+1).
\end{equation}
\end{lemma}

\begin{proof}
Since $f_k(t)=(1+u_kt)f_{k-1}(t)=f_{k-1}(t)+u_k\,t\,f_{k-1}(t)$, we have
\[
f_k(t)=\sum_{r=0}^{k-1}f_{k-1,r}\,t^r+u_k\sum_{r=0}^{k-1}f_{k-1,r}\,t^{r+1},
\]
and therefore, by the definition of $Q_k$,
\[
Q_k(x)=\sum_{r=0}^{k-1}f_{k-1,r}\,(x)_r\bigl(x+r+\tfrac32\bigr)_{k-r}+u_k\sum_{r=0}^{k-1}f_{k-1,r}\,(x)_{r+1}\bigl(x+r+\tfrac52\bigr)_{k-r-1}.
\]
For each $(x+r+\tfrac32)_{k-r}$ in the first sum we remove the final factor \[x+r+\tfrac32+(k-r-1)=x+k+\half.\]
In the second sum we write \[(x)_{r+1}=x\,(x+1)_r \quad\text{and}\quad x+r+\tfrac52=(x+1)+r+\tfrac32.\] This yields
\[
Q_k(x)=\bigl(x+k+\half\bigr)\sum_{r=0}^{k-1}f_{k-1,r}\,(x)_r\bigl(x+r+\tfrac32\bigr)_{k-r-1}
+u_k\,x\sum_{r=0}^{k-1}f_{k-1,r}\,(x+1)_r\bigl((x+1)+r+\tfrac32\bigr)_{k-r-1}.\]
The first sum is $Q_{k-1}(x)$ and the second sum is $Q_{k-1}(x+1)$.
\end{proof}

\subsection{The induction}\label{sec:induction}
We use the following elementary fact: if a real polynomial has positive leading coefficient and only simple real roots, then at any point $\xi$ which is not a root its sign is $(-1)^N$, where $N$ is the number of roots greater than $\xi$.

\begin{lemma}\label{lem:induction}
For every $1\le k\le d$, the polynomial $Q_k$ has degree $k$ and positive leading coefficient $f_k(1)=\prod_{j=1}^k(1+u_j)$, and it has $k$ simple real roots, which lie in the interval $(-k-\half,0)$ and are pairwise more than $1$ apart.
\end{lemma}

\begin{proof}
Each product $(x)_r(x+r+\frac32)_{k-r}$ in \eqref{eq:Qk} is monic of degree $k$, so the coefficient of $x^k$ in $Q_k$ is $\sum_rf_{k,r}=f_k(1)>0$. We prove the statement about the roots by induction on $k$.

For $k=1$ we have $Q_1(x)=(1+u_1)x+\tfrac32$, whose only root $-\tfrac{3}{2(1+u_1)}$ lies in $(-\tfrac32,0)$.

Now let $2\le k\le d$ and suppose the statement holds for $Q_{k-1}$, with roots
\[
\rho_1<\rho_2<\dots<\rho_{k-1},\qquad -(k-1)-\half<\rho_1,\quad \rho_{k-1}<0,\quad \rho_{i+1}-\rho_i>1 .
\]
Write the two terms of the recurrence \eqref{eq:Qrec} as
\[
A(x)=\bigl(x+k+\half\bigr)Q_{k-1}(x),\qquad B(x)=x\,Q_{k-1}(x+1),\qquad\text{so that}\qquad Q_k=A+u_kB.
\]
Both have degree $k$ and the positive leading coefficient $f_{k-1}(1)$. To avoid a clash with the coefficients $a_k$ and the numbers $b_j$, we denote the roots of $A$ by $\alpha_i$ and the roots of $B$ by $\beta_i$.
\[
\alpha_1=-k-\half,\quad \alpha_{i+1}=\rho_i\ \ (1\le i\le k-1);\qquad
\beta_i=\rho_i-1\ \ (1\le i\le k-1),\quad \beta_k=0 .
\]
The roots of $A$ and $B$ alternate:
\begin{equation}\label{eq:alternate}
\alpha_1<\beta_1<\alpha_2<\beta_2<\dots<\alpha_k<\beta_k.
\end{equation}
Indeed, the leftmost inequality is the bound $\rho_1>-(k-1)-\half$ on the smallest old root, the inequalities $\rho_i-1<\rho_i$ are trivial, the inequalities $\rho_i<\rho_{i+1}-1$ are the gaps greater than one, and the rightmost inequality is $\rho_{k-1}<0$. In particular $A$ and $B$ have simple real roots and no common root.

At the points $\alpha_i$ and $\beta_i$ we have
\[
Q_k(\alpha_i)=u_k\,B(\alpha_i),\qquad Q_k(\beta_i)=A(\beta_i).
\]
These values are nonzero and have opposite signs. Because $B$ has $k-i+1$ roots to the right of $\alpha_i$ (namely $\beta_i,\dots,\beta_k$), and $A$ has $k-i$ roots to the right of $\beta_i$ (namely $\alpha_{i+1},\dots,\alpha_k$), by \eqref{eq:alternate},
\[
\sgn Q_k(\alpha_i)=(-1)^{k-i+1},\qquad \sgn Q_k(\beta_i)=(-1)^{k-i}.
\]
The intermediate value theorem therefore gives a root $\sigma_i$ of $Q_k$ in each of the intervals
\[
(\alpha_i,\beta_i):\qquad \bigl(-k-\half,\ \rho_1-1\bigr),\quad(\rho_1,\ \rho_2-1),\quad\dots,\quad(\rho_{k-1},\ 0).
\]
There are $k$ such intervals and they are pairwise disjoint. Since $Q_k$ has degree $k$, the numbers $\sigma_1<\dots<\sigma_k$ are all the roots of $Q_k$, and each is simple. Notably, they all lie in $(-k-\half,0)$. Consecutive intervals are separated by the gaps $[\beta_i,\alpha_{i+1}]=[\rho_i-1,\rho_i]$ of length exactly one, so
\[
\sigma_{i+1}-\sigma_i>\alpha_{i+1}-\beta_i=1 .
\]
This completes the induction.
\end{proof}

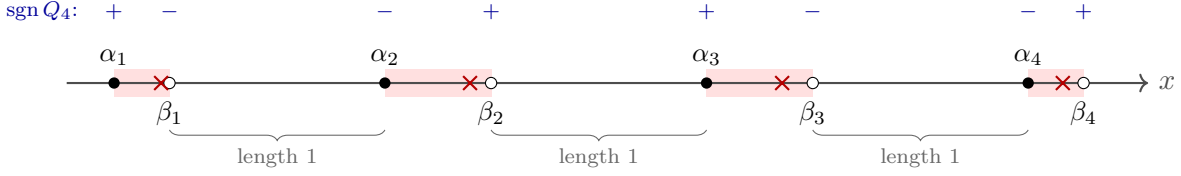
\begin{figure}[t]
\centering
\begin{tikzpicture}[x=2.85cm,y=1cm,
  Adot/.style={circle,fill=black,inner sep=1.5pt},
  Bdot/.style={circle,draw=black,fill=white,inner sep=1.5pt},
  newz/.style={cross out,draw=red!70!black,thick,inner sep=2.2pt}]
\def\ao{-4.5}\def\bo{-4.2430}\def\at{-3.2430}\def\bt{-2.75}
\def\ah{-1.75}\def\bh{-1.2570}\def\af{-0.2570}\def\bff{0}
\foreach \a/\b in {\ao/\bo,\at/\bt,\ah/\bh,\af/\bff}
  \fill[red!12] (\a,-0.18) rectangle (\b,0.18);
\draw[->,thick,black!70] (-4.72,0) -- (0.3,0) node[right]{$x$};
\foreach \a/\b in {\bo/\at,\bt/\ah,\bh/\af}
  \draw[decorate,decoration={brace,mirror,amplitude=4pt},black!60] (\a,-0.62) -- node[below=3pt,font=\scriptsize]{length $1$} (\b,-0.62);
\foreach \p/\l/\s in {\ao/1/+,\at/2/-,\ah/3/+,\af/4/-}
  {\node[Adot] at (\p,0) {}; \node[above=3pt,font=\small] at (\p,0) {$\alpha_{\l}$};
   \node[font=\scriptsize,text=blue!60!black] at (\p,0.95) {$\s$};}
\foreach \p/\l/\s in {\bo/1/-,\bt/2/+,\bh/3/-,\bff/4/+}
  {\node[Bdot] at (\p,0) {}; \node[below=3pt,font=\small] at (\p,0) {$\beta_{\l}$};
   \node[font=\scriptsize,text=blue!60!black] at (\p,0.95) {$\s$};}
\foreach \s in {-4.2818,-2.8479,-1.3991,-0.0962} \node[newz] at (\s,0) {};
\node[font=\scriptsize,text=blue!60!black,anchor=east] at (-4.62,0.95) {$\sgn Q_4$:};
\end{tikzpicture}
\caption{The inductive step of \cref{lem:induction} for $k=4$.}
\label{fig:interlace}
\end{figure}

\subsection{Conclusion}\label{sec:conclusion}

\begin{proof}[Proof of \cref{thm:main}]
By \eqref{eq:R0} and \eqref{eq:RQ}, $P_n(x)=c_0\,2^{n-2}\,(x+1)_{m-1+\varepsilon}\,R(x)$ with $R=h(0)\,Q_d$.
If $d=0$, then $R=h(0)$ is a positive constant (\cref{sec:h}). If $d\ge1$, then by \cref{lem:induction} with $k=d$, the polynomial $R=h(0)Q_d$ has degree $d$, leading coefficient $h(0)f_d(1)=h(1)>0$, and $d$ simple real roots in $(-d-\half,0)=(-(m-\half),0)$ that are pairwise more than one unit apart. This proves (b).

Part (a) follows, since $\deg P_n=(m-1+\varepsilon)+(m-1)=n-2$ and the leading coefficient of $P_n$ is $c_0\,2^{n-2}\,h(1)>0$. For (c), the roots of $(x+1)_{m-1+\varepsilon}$ are $-1,\dots,-(m-1+\varepsilon)$, and both $m-1+\varepsilon$ and $m-\half$ are at most $\frac{n-1}{2}=m+\frac{\varepsilon-1}{2}$. Hence all roots of $P_n$ lie in $[-\frac{n-1}{2},0)$. For odd $n$ the endpoint $-\frac{n-1}{2}=-m$ is the last root of $(x+1)_m$.

\end{proof}

\begin{proof}[Proof of \cref{cor:PF}]
By \cref{thm:main}, $P_n(x)=\kappa\prod_{i=1}^{n-2}(x+\theta_i)$ with $\kappa>0$ and $\theta_i>0$, so the coefficients of $P_n$ are $\kappa$ times elementary symmetric functions of the $\theta_i$ and are positive. By \cite{ASW,Ed}, the coefficient sequence is $\PF$. Finally, the roots of $Q_n^{1,1}(y)=P_n(y+1)$ are the numbers $\theta-1$, where $\theta\in[-\frac{n-1}2,0)$ runs over the roots of $P_n$.
\end{proof}

\begin{remark}[Low degrees]
Comparing the coefficients of $z$ on both sides of \eqref{eq:Cg} gives $c_1=c_0(n+g_1)$. Equivalently,
\begin{equation}\label{eq:g1}
c_0\,g_1=n\bigl(a_1a_{n-1}-a_0a_n\bigr).
\end{equation}
Evaluating \eqref{eq:R0} at $x=0$, where only the term $r=0$ of $R$ survives, and using $2^{n-2}(m-1+\varepsilon)!\,(\tfrac32)_{m-1}=(n-1)!$, one recovers Karp's formula $P_n(0)=n!\,(a_1a_{n-1}-a_0a_n)$ \cite[(14)]{Ka12}.

\Cref{tab:low} lists $P_n$ for $2\le n\le5$. In this range $m\le2$, so $R$ is either the constant $g_1$ or, for $n=4,5$, the linear polynomial $g_1(x+\tfrac32)+\tfrac12g_2x$. To pass to the coefficients of $p$, put $\Delta_k=a_ka_{n-k}-a_0a_n$, so that $c_k=\binom nk(c_0+\Delta_k)$. Then \eqref{eq:g1} reads $c_0g_1=n\Delta_1$, and comparing the coefficients of $z^2$ in \eqref{eq:Cg} gives $c_0g_2=\binom n2\Delta_2-n(n-2)\Delta_1$ for $n\ge4$. All entries were checked by exact symbolic computation.

\begin{table}[ht]
\centering
\renewcommand{\arraystretch}{1.35}
\small
\begin{tabular}{@{}c l l@{}}
\toprule
$n$ & $P_n(x)$ in terms of $g_1,g_2$ & $P_n(x)$ in terms of $\Delta_1,\Delta_2$ \\
\midrule
$2$ & $c_0g_1$ & $2\Delta_1$ \\
$3$ & $2c_0g_1\,(x+1)$ & $6\Delta_1\,(x+1)$ \\
$4$ & $4c_0\,(x+1)\bigl[g_1(x+\tfrac32)+\tfrac12g_2x\bigr]$ & $12\,(x+1)\bigl[\Delta_2\,x+2\Delta_1\bigr]$ \\
$5$ & $8c_0\,(x+1)(x+2)\bigl[g_1(x+\tfrac32)+\tfrac12g_2x\bigr]$ & $20\,(x+1)(x+2)\bigl[(2\Delta_2-\Delta_1)\,x+3\Delta_1\bigr]$ \\
\bottomrule
\end{tabular}
\caption{The polynomials $P_n$ for small $n$, where $\Delta_k=a_ka_{n-k}-a_0a_n$.}
\label{tab:low}
\end{table}

For instance, when $n=4$ the only root of $R$ is $x^*=-3g_1/(2g_1+g_2)$, and the statement $x^*\in(-\tfrac32,0)$ of \cref{thm:main} amounts to $g_1>0$ and $g_2=b_1b_2>0$, i.e.\ to the inequalities $a_1a_3>a_0a_4$ and $3a_2^2+a_0a_4>4a_1a_3$ for real-rooted quartics with positive coefficients.
\end{remark}

\section{Concluding remarks}\label{sec:remarks}

\Cref{conj:karp} is stated for $\PF$ sequences, and finite $\PF$ sequences may have $a_0=0$ or $a_n=0$ (their generating polynomials then have a root at the origin, or degree less than $n$). In that case the word ``negative'' must be replaced by ``nonpositive'', as the following example and proposition show.

\begin{example}
For $(a_0,\dots,a_4)=(0,0,1,2,1)$, i.e.\ $p(z)=z^2(1+z)^2$, one finds
\[
P_4(x)=12\,x(x+1),
\]
which vanishes at $x=0$. For $(a_0,\dots,a_3)=(0,0,1,1)$ all products $a_ka_{3-k}$ vanish and $P_3\equiv0$.
\end{example}

\begin{proposition}\label{prop:degenerate}
Let $n\ge2$ and let $a_0,\dots,a_n\ge0$, not all zero, be such that $p(z)=\sum_ka_kz^k$ has only real roots. Then either $P_n\equiv0$, or $P_n$ has degree $n-2$ and all its roots are real and lie in $[-\frac{n-1}{2},0]$. Moreover, $0$ is a root of $P_n$ if and only if $a_1a_{n-1}=a_0a_n$.
\end{proposition}

\begin{proof}
Write $p(z)=\lambda\,z^\mu\prod_{i=1}^{\nu}(z+r_i)$ with $\lambda>0$, $r_i>0$ and $\mu+\nu=\deg p\le n$. For $\delta>0$ put  \[p_\delta(z)=\lambda\,(z+\delta)^\mu(1+\delta z)^{n-\mu-\nu}\prod_i(z+r_i).\] Then $p_\delta$ has degree $n$, positive coefficients and only real roots, and its coefficients tend to $a_0,\dots,a_n$ as $\delta\to0^+$. Since the coefficients of $P_n$ are polynomial functions of $a_0,\dots,a_n$, the corresponding polynomials $P_n^{(\delta)}$ converge coefficientwise to $P_n$. By \cref{thm:main}, $P_n^{(\delta)}(x)=\kappa_\delta\prod_{i=1}^{n-2}(x-x_i(\delta))$ with $\kappa_\delta>0$ and $x_i(\delta)\in[-\frac{n-1}2,0)$. The leading coefficients $\kappa_\delta$ converge to the coefficient $\kappa_0\ge0$ of $x^{n-2}$ in $P_n$. Choosing a sequence $\delta_l\to0$ along which every $x_i(\delta_l)$ converges to some $x_i\in[-\frac{n-1}2,0]$, we obtain $P_n(x)=\kappa_0\prod_i(x-x_i)$. If $\kappa_0=0$ then $P_n\equiv0$; otherwise $P_n$ has the asserted form. The last statement follows from Karp's formula $P_n(0)=n!\,(a_1a_{n-1}-a_0a_n)$ \cite[(14)]{Ka12}, which holds for arbitrary coefficients.
\end{proof}

Since the polynomials $P_n^r$ of \cite[\S4]{Ka12}, defined from the $r\times r$ Toeplitz determinant $\det\bigl(f(x+j-i;z)\bigr)_{i,j=1}^r$, satisfy $P_n^2=P_n$, \cref{thm:main} and \cref{cor:PF} also prove the case $r=2$ of \cite[Conjectures~5 and~6]{Ka12} for positive $\PF$ sequences. The same conjectures for $r\ge3$, and \cite[Conjectures~1 and~2]{Ka12} for log-concave sequences and general shifts $\alpha,\beta>0$, are not addressed here. (Karp's polynomials $Q_n^{\alpha,\beta}$ carry superscripts and should not be confused with the polynomials $Q_k$ of \cref{sec:recurrence}.) The approach above relies on the bilinear structure of the $2\times2$ determinant and on the symmetric shifts $\pm1$, which make $\Lop$ computable in closed form on the polynomials $z^j(z+1)^{n-2j}$ (\cref{lem:F}). It would be interesting to find analogous decompositions for the multilinear expressions arising for $r\ge3$.

\subsection*{Use of AI tools}
The author used GPT-6-Astra to explore proof strategies, to help obtain an earlier proof of \cref{conj:karp} that differs from the one given here, to help draft the text, and to write the Python code for the computational
checks; Claude helped design the figures to the author's specifications. The proof presented here is the author's own, and the author takes full responsibility for the contents of the paper.

\subsection*{Acknowledgements}
The author thanks Anna Dietrich, Damien Galant, and Javier G\'omez-Serrano for their mentorship and support. This paper and the results within it were made possible by the generosity and vision of an anonymous donor.


\end{document}